\documentclass[11pt]{article}
\usepackage{amsfonts,amsmath,amssymb}
\usepackage{url}
\usepackage{epsfig,latexsym,graphicx}
\usepackage{esint}
\usepackage{graphicx,latexsym,amssymb,amsmath,amsfonts,fancyhdr,mathrsfs}
\usepackage{euscript}
\newtheorem{maintheorem}{Theorem}
\newtheorem{maincoro}[maintheorem]{Corollary}

\newtheorem{theorem}{Theorem}[section]
\newtheorem{lemma}[theorem]{Lemma}
\newtheorem{claim}[theorem]{Claim}

\newtheorem{corollary}[theorem]{Corollary}

\newenvironment{proof}{\noindent{\bf Proof.\,\ }}{\hfill\mbox{$\Box$}\smallskip}

\def\XXint#1#2#3{{\setbox0=\hbox{$#1{#2#3}{\int}$ }
\vcenter{\hbox{$#2#3$ }}\kern-.6\wd0}}

\def\P{\mathbb{P}}
\def\Z{\mathbb{Z}}

\def\E{\mathbb{E}}

\newcommand{\cD}{\mathcal{D}}
\newcommand{\cF}{\mathcal{F}}

\newcommand{\fR}{\mathfrak{R}}

\newcommand{\sA}{\mathscr{A}}
\newcommand{\sB}{\mathscr{B}}

\newcommand{\sD}{\mathscr{D}}
\newcommand{\sH}{\mathscr{H}}
\newcommand{\sR}{\mathscr{R}}
\newcommand{\sW}{\mathscr{W}}

\begin{document}

\title{On One-dimensional
Multi-Particle Diffusion Limited Aggregation}

\author{
Allan Sly
\thanks{Princeton University and University of California, Berkeley. Supported  by NSF grant DMS-1352013. Email:asly@math.princeton.edu}
}
\date{}

\maketitle
\begin{abstract}
We prove that the one dimensional Multi-Particle Diffusion Limited Aggregation model has linear growth whenever the particle density exceeds 1 answering a question of Kesten and Sidoravicius.  As a corollary we prove linear growth in all dimensions~$d$ when the particle density is at least 1.
\end{abstract}

\section{Introduction}
In the Diffusion Limited Aggregation (DLA) model introduced by Witten and Sanders~\cite{WitSan:81} particles arrive from infinity and adhere to a growing aggregate.  It produces beautiful fractal-like pictures of dendritic growth but mathematically it remains poorly understood.  We consider a variant, multiparticle DLA, where the aggregate sits in an infinite Poisson cloud of particles which adhere when they hit the aggregate, a model which has been studied in both physics~\cite{Voss:84} and mathematics~\cite{KesSid:08,SidSta:16}.  Again one is interested in the growth of the aggregate and its structure.

In the model, initially there is a collection of particles whose locations are given by a mean $K$ Poisson initial density on $\Z^d$.  The particles each move independently according to rate 1 continuous time random walks on $\Z^d$.  We follow the random evolution of an aggregate $\cD_t \subset \Z^d$ where at time 0 an aggregate is placed at the origin $\cD_0=\{0\}$ to which other particles adhere according the the following rule.  When a particle at $v\not\in \cD_{t-}$ attempts to move onto the aggregate $\cD_t$ at time $t$, it stays in place and instead is added to the aggregate so $\cD_t=\cD_{t-} \cup \{v\}$ and the particle no longer moves.  Any other particles at $v$ at the time are also frozen in place.

We will mainly focus on the one dimension setting and in Section~\ref{s:Higher} will discuss how to boost the results to higher dimensions.  In this case the aggregate is simply a line segment and the processes on the positive and negative axes are independent so we simply restrict our attention to the rightmost position of the aggregate at time $t$ which we denote $X_t$.  In this case at time $t$ when a particle at $X_{t-} +1$ attempts to take a step to the left it is incorporated into the aggregate along with any other particles.

It was proved by Kesten and Sidoravicius~\cite{KesSid:08} that $X_t$ grows like $\sqrt{t}$ when $K<1$.  Indeed there simply are not enough particles around for it to grow faster.  They conjectured, however, that when $K>1$ then it should grow linearly.  Our main result confirms this conjecture.

\begin{maintheorem}\label{t:mainThm}
For all $K>1$ the limit $\lim_t \frac1{t} X_t$ exists almost surely and is a positive constant.
\end{maintheorem}

We also give a simple extension of these results to higher dimensions and prove the following corollary.

\begin{maincoro}\label{c:Higher}
In all dimensions $d\geq 2$ when $K > 1$ the diameter of the aggregate  grows linearly in $t$, that is for some positive constant $\delta>0$
\[
\lim_t \frac1{t} \hbox{Diam}(\cD_t) > \delta \ \hbox{a.s.}
\]
\end{maincoro}

Previously  Sidoravicius and  Stauffer~\cite{SidSta:16} studied the the case of $d\geq 2$ in a slightly different variant where particles instead perform a simple exclusion process.  They showed that for densities close to 1, that there is a positive probability that the aggregate grows with linear speed.  Also in Section~\ref{s:Higher} we describe how for $d\geq 2$ the upper bound on the threshold can be reduced further below 1, for example to $\frac56$ when $d=2$.  However, strikingly Eldan~\cite{Eldan:16} conjectured that the critical value is always 0, that is the aggregate grows with linear speed for all $K>0$.  We are inclined to agree with this conjecture but our methods do not suggest a way of reaching the threshold.  A better understanding of the growth of the standard DLA seems to be an important starting point.

\section{Basic results}
We will analyse the function valued process $Y_t$ given by,
\begin{equation}\label{e:yDefn}
Y_t(s) := \begin{cases}
X_t - X_{t-s} & 0 \leq s \leq t\\
\infty & s > t.
\end{cases}
\end{equation}
Let $\cF_t$ denote the filtration generated by $X_t$. We let $S(t)$ denote the infinitesimal rate at which $X_t$ increases given $\cF_t$.  Given $\cF_t$ the number of particles at $X_t+1$ is conditionally Poisson with intensity given by the probability that a random walker at $X_t+1$ at time $t$ was never located in the aggregate.  Each of the particles jumps to the left at rate $\frac12$ so with $W_t$ denoting an independent continuous time random walk,
\[
S(t) = \frac12 K \P[\max_{0\leq s \leq t} W_s - Y_t(s) \leq 0\mid Y_t].
\]
Note that $S(t)$ is an increasing as a function of $Y_t$.  Indeed we could realise $X_t$ as follows, let $\Pi$ be a Poisson process on $[0,\infty)^2$ and then \[
X_t = \Pi(\{(x,y): 0 \leq x \leq t, 0 \leq y \leq S(x)\}.
\]
Since both $X_t$ and $Y_t$ are increasing functions of $\Pi$ we can make use of the FKG property.  Also note that $Y_t$ is stochastically decreasing.

Most of our analysis will involve estimating $S(t)$ and using that to show that $Y_t$ does not become too small for too long.  Let $M_t = \max_{0\leq s \leq t} W_s$ be the maximum process of $W_t$.

\begin{lemma}\label{l:SDecomp}
For any $i\geq 0$ we have that
\[
S(t) \geq \frac{K}{2}\P[M_{2^i} = 0]\prod_{i'=i}^\infty \P[M_{2^{i'+1}} \leq Y_t(2^{i'}) \mid Y_t]
\]
\end{lemma}
\begin{proof}
We have
\begin{align*}
S(t) &\geq \frac{K}2 \P[\max_{0\leq s \leq t} W_s - Y_t(s) \leq 0\mid Y_t]\\
&\geq \frac{K}2  \P[M_{2^i} = 0, \forall i' \geq i \ M_{2^{i'+1}} \leq Y_t(2^{i'})\mid Y_t] \\
&\geq \frac{K}2 \P[M_{2^{i}} = 0]\prod_{i' \geq i} \P[M_{2^{i'+1}} \leq Y_t(2^{i'})\mid Y_t]
\end{align*}
where the final inequality follows from the FKG inequality.
\end{proof}

By the reflection principle we have that for any integer $j \geq 0$,
\[
\P[M_t \geq  j] = \P[W_t \geq j] + \P[W_t \geq j+1].
\]
Thus asymptotically we have that
\begin{equation}\label{e:MTasymptotics}
\P[M_t=0] \approx \frac1{\sqrt{2\pi}} t^{-{1/2}}
\end{equation}
Now let $T_j$ be the first hitting time of $j$.  Since $\cosh(s)-1 \leq s^2$ for $0 \leq s \leq 1$ we have that for $t \geq 1$,
\[
\E[e^{\frac1{\sqrt{t}}W_{t \vee T_j}}] \leq \E[e^{\frac1{\sqrt{t}}W_{t}}] = e^{(\cosh(\frac1{\sqrt{t}})-1)t}\leq e^{1},
\]
and hence by Markov's inequality
\begin{equation}\label{e:reflectionBound}
\P[M_t \geq j t^{1/2}] \leq \P[e^{\frac1{\sqrt{t}}W_{t \vee T_j}}=e^{j}] \leq e^{1-j}.
\end{equation}
Plugging the above equations into Lemma~\ref{l:SDecomp} we get the following immediate corollary.
\begin{corollary}\label{c:SLowBound}
There exists $i^*$ such that the following holds.  Suppose that $i \geq i^*$ that for all $i'\geq i$ we have $j_{i'} = Y_t(2^{i'})2^{-{i'}/2}$.  Then
\[
S(t) \geq \frac{K}{10}2^{-{i'}/2}\prod_{i'=i}^\infty (1-e^{1-\max\{1,j_{i'}/\sqrt{2}\}})
\]
\end{corollary}


Next we check that provided $S(t)$ remains bounded below during an interval then we get a comparable lower bound on the speed of $X_t$.
\begin{lemma}\label{l:minSpeed}
We have that for all $\rho \in (0,1)$ there exists $\psi(\rho)>0$ such that for all $\Delta >0$,
\[
\P[\min_{s\in[t,t+\Delta]} S(s) \geq \gamma, X_{t+\Delta}-X_t \leq \rho \Delta \gamma \mid Y_t] \leq \exp(-\psi(\rho)\Delta \gamma)
\]
In the case of $\rho=\frac12$ we have $\psi(\rho) \geq \frac1{10}$.
\end{lemma}
\begin{proof}
Using the construction of the process in terms of $\Pi$ we have that
\begin{align*}
\P[\min_{s\in[t,t+\Delta]} S(s) \geq \gamma, X_{t+\Delta}-X_t \leq \rho \Delta \gamma \mid Y_t] &\leq \P[\Pi([t,t+\Delta]\times[0,\gamma]) \leq \rho \Delta \gamma]\\
&=\P[\hbox{Poisson}(\Delta \gamma) \leq \rho \Delta \gamma]
\end{align*}
Now if $N\sim \hbox{Poisson}(\Delta \gamma)$ then $\E e^{-\theta N} = \exp((e^{-\theta}-1)\Delta \gamma)$ and so by Markov's inequality
\[
\P[N \leq \rho \Delta \gamma] = \P[e^{-\theta N} \geq e^{- \theta \rho \Delta \gamma}] \leq \frac{\exp((e^{-\theta}-1)\Delta \gamma)}{\exp(-\theta \rho \Delta \gamma)} = \exp((\theta\rho+e^{-\theta}-1)\Delta \gamma).
\]
Setting $f_\rho(\theta) = -(\theta\rho+e^{-\theta}-1)$ and
\[
\psi(\rho) = \sup_{\theta \geq 0} f_\rho(\theta)
\]
it remains to check that $\psi(\rho)>0$.  This follows from the fact that $f_\rho(0)=0$ and $f^{'}_\rho(0) = 1-\rho > 0$.  Since $f_{\frac12}(\frac12) \geq \frac1{10}$ we have that $\psi(\frac12) \geq \frac1{10}$.
\end{proof}

\section{Proof of Positive Speed}

To measure our control over $Y_t$ and show that it is moving quickly enough we say that $Y_t$ is permissive at time $t$ and at scale $i$ if $Y_t(2^i) \geq  10 i 2^{i/2}$. Our approach, will be to consider functions
\[
y_\alpha(s)=\begin{cases}
0 &s  \leq \alpha^{-3/2}\\
\min\{\alpha (s-\alpha^{-3/2}), s^{1/2}\log_2 s\} & s \geq \alpha^{-3/2}.
\end{cases}
\]
and show that if $Y_t(s) \geq y_\alpha(s)$ for increasing values of $\alpha$ with good probability.  To measure the speed of the aggregate in an interval of time define events $\sR$ as
\[
\sR(t,s,\gamma) =\{X_{t+s}-X_t \geq \gamma s\}.
\]

\begin{lemma}\label{l:RWYalpha}
For all $\epsilon>0$ there exists $0<\alpha_\star(\epsilon) \leq 1$ such that for all $0<\alpha<\alpha_\star$,
\[
\P[\max_{s \geq 0} W_s - y_\alpha((s-\alpha^{-4/3})\wedge 0) \leq 0] \geq 2(1-\epsilon)\alpha.
\]
\end{lemma}
\begin{proof}
For small $\alpha_\star(\epsilon)$ we have that for $\alpha^{-3/2} \leq s\leq \alpha^{-2}$,
\[
\alpha (s-\alpha^{-4/3} -\alpha^{-3/2}) \leq (s-\alpha^{-4/3})^{1/2}\log_2 (s-\alpha^{-4/3})
\]
Hence with $\xi=\xi_\alpha=\alpha^{4/3}+\alpha^{-3/2}$ if we set
\[
\sA=\{\max_{s \geq 0} W_s - \alpha((s-\xi)\wedge 0) \leq 0\}
\]
and
\[
\sB=\{\max_{s \geq \alpha^{-2}} W_s - (s-\alpha^{-4/3})^{1/2}\log_2 (s-\alpha^{-4/3}) \leq 0\}
\]
then
\[
\P[\max_{s \geq 0} W_s - y_\alpha((s-\alpha^{4/3})\wedge 0) \leq 0] \geq \P[\sA,\sB] \geq \P[\sA]\P[\sB].
\]
where the second inequality follows by the FKG inequality since $\sA$ and $\sB$ are both decreasing events for $W_s$.  For large $s$, we have $s^{1/2}\log_2 s \leq 2 (s/2)^{1/2} \log_2 (s/2)$ and so
\begin{align*}
\P[\sB] &\geq \P[\max_{s \geq \alpha^{-2}} W_s - \tfrac12 s^{1/2}\log_2 (\tfrac12 s) \leq 0]\\
&\geq \P[\forall i \geq \lfloor \log_2 (\alpha^{-2})\rfloor M_{2^{i+1}} \leq \tfrac12 (i-1) 2^{i/2} ]\\
&\geq \prod_{i \geq \lfloor \log_2 (\alpha^{-2})\rfloor} \P[M_{2^{i+1}} \leq \tfrac12 (i-1) 2^{i/2} ]\\
&\geq \prod_{i \geq \lfloor \log_2 (\alpha^{-2})\rfloor} e^{1-(i-1)2{-3/2}}
\end{align*}
where the third inequality follows from the FKG inequality and the final inequality is by equation~\eqref{e:reflectionBound}.  Thus as $\alpha\to 0$ we have that $\P[\sB] \to 1$ so it is sufficient to show that for small enough $\alpha$  that $\P[\sA] \geq 2\alpha(1-\epsilon/2)$.  By the reflection principle for $a \leq 1$,
\[
\P[M_t \geq 1, W_t = a] = \P[M_t \geq 1, W_t = 2-a]=\P[ W_t = 2-a] = \P[W_t = a-2],
\]
and so
\begin{align*}
\P[M_t \geq 1]&=\sum_{a> 1} \P[M_t \geq 1, W_t = a] + \sum_{a \leq 1} \P[M_t \geq 1, W_t = a]\\
&=\sum_{a> 1} \P[M_t \geq 1, W_t = a] + \sum_{a \leq 1} \P[W_t = a-2]\\
&= 1- \P[W_t=0]-\P[W_t=1].
\end{align*}
Hence by the Local Central Limit Theorem,
\[
\lim_t \sqrt{t}\P[M_t = 0] = \lim_t \sqrt{t}\left(\P[W_t=0]+\P[W_t=1]\right) = \frac{2}{\sqrt{2\pi}}.
\]
Also we have for $a\leq 0$,
\[
\P[W_t=a, M_t=0] = \P[W_t=a] - \P[W_t=a, M_t \geq 1] = \P[W_t = a] - \P[W_t = a-2]
\]
and so the law of $W_t$ conditioned on $M_t=0$ satisfies,
\begin{align*}
\lim_t\P[\frac1{\sqrt{t}} W_t \leq x \mid M_t=0] &= \lim_t\frac{\sum_{a=-\infty}^{x\sqrt{t}}\P[W_t = a] - \P[W_t = a-2]}{\P[M_t = 0]}\\
&=\lim_t \frac{\P[W_t = x\sqrt{t}] +\P[W_t = x\sqrt{t}-1]}{\P[M_t = 0]}\\
&=\lim_t\frac{2\frac{1}{\sqrt{2\pi}} e^{-x^2/2}}{\frac{2}{\sqrt{2\pi}}}= e^{-x^2/2}
\end{align*}
where $x \leq 0$ and hence is the negative of the Rayleigh distribution.  Now let $Z_t = W_t - \alpha t$ and $U_t = e^{\theta Z_t}$.  Then
\[
\E U_t = \exp((\cosh(\theta)-1-\alpha\theta)t).
\]
As $f_\alpha(\theta) = \cosh(\theta)-1-\alpha\theta$ is strictly convex, it has two roots, one of which is at $\theta=0$.  Let $\theta_\alpha$ be the non-zero root of $f_\alpha$.  Since
\[
f_\alpha(\theta)=-\alpha\theta+\frac12 \theta^2 + O(\theta^4)
\]
for small $\alpha$ we have that $\theta_\alpha=2\alpha + O(\alpha^2)$.  Then with $\theta=\theta_\alpha$ we have that $U_t = e^{\theta_\alpha Z_t}$ is a martingale.  Let $T = \min_t Z_t > 0$ and so by the Optional Stopping Theorem,
\[
\E[U_T\mid Z_T=z] = \E[U_0\mid Z_T=z] = e^{\theta_\alpha x}.
\]
Also since $U_T\in[0,1]$ if $T < \infty$ so
\[
\E[U_T\mid Z_T=z] \geq \P[T < \infty\mid Z_T=z]
\]
so
\[
\P[T < \infty\mid Z_T=z] \leq e^{-\theta_\alpha}.
\]
Thus we have that as $\alpha\to 0$,
\begin{align*}
\P[\sA] &=\P[\max_{s \geq 0} W_s - \alpha((s-\xi)\wedge 0) \leq 0]\\
&= \sum_{x=-\infty}^{0} \P[M_\xi = 0, W_\xi=x] \P[T = \infty\mid Z_T=z]\\
&\geq \P[M_\xi = 0] \sum_{x=-\infty}^{0} \P[W_\xi=x \mid M_\xi = 0](1-e^{\theta_\alpha x})\\
&\geq \frac{2+o(1)}{\sqrt{2\pi t}} \sum_{x=-\infty}^{0} \P[W_\xi=x \mid M_\xi = 0] (-2\alpha x)\\
&\to  \frac{4\alpha}{\sqrt{2\pi t}}\sqrt{\frac{\pi}{2}} = 2\alpha,
\end{align*}
since the mean of the Rayleigh distribution is $\sqrt{\frac{\pi}{2}}$.  This completes the lemma.
\end{proof}

\begin{lemma}\label{l:initialYalpha}
For all $K>1$ there exists $i_\star(K)$ such that if $i \geq i_\star$ and $Y_T$ is permissive at all levels $i$ and above then with
\[
\alpha=\frac1{80} 2^{-i/2}
\]
we have that
\[
\P[\inf_s Y_{T+2^{i}}(s) - y_{\alpha}(s) \geq 0\mid Y_T] \geq 1-\exp(-2^{i/10}).
\]
\end{lemma}
\begin{proof}
Since $Y_T(2^{i'}) \geq 10 i' 2^{i'/2}$ for all $i' \geq i$ if we set
\[
\tilde{y}(s) = \begin{cases}
0 &s < 2^{i+1}\ , \\
10 j 2^{j/2} &s\in [2^{j+1},2^{j+2}), \ j \geq i.
\end{cases}
\]
then since $Y_{T+u}(s) \geq Y_T(s-u)$ then
\begin{equation}\label{e:initialYalpha1}
\inf_{0\leq u \leq 2^{i}} \inf_{s \geq 0} Y_{T+u}(s) - \tilde{y}(s) \geq 0.
\end{equation}
By Corollary~\ref{c:SLowBound} for all $t\in [0,2^i]$
\[
S(t) \geq \frac{1}{10}2^{-(i+1)/2}\prod_{i'=i}^\infty (1-e^{1-\max\{1,5(i'-1)\}}) \geq \frac1{20} 2^{-i/2} \ ,
\]
where the second inequality holds provided that $i_\star(K)$ is sufficiently large.  Defining $\sD$ as the event that $X_t$ moves at rate at least $\frac1{40} 2^{-i/2}$ for each interval $\ell 2^{2i/3},(\ell+1)2^{2i/3}]$,
\[
\sD=\bigcap_{\ell=0}^{2^{i/3}-1} \sR(T+\ell 2^{2i/3},2^{2i/3},\frac1{40} 2^{-i/2})
\]
by Lemma~\ref{l:minSpeed} we have that
\[
\P[\sD] \geq  1 - 2^{i/3}\exp(-\frac1{10}\cdot 2^{2i/3} \cdot \frac1{40} 2^{-i/2})\geq 1-\exp(-2^{i/10})
\]
where the last inequality holds provided that $i_\star(K)$ is sufficiently large.  We claim that on the event $\sD$, we have that $Y_{T+2^{i}}(s) \geq y_{\alpha}(s)$ for all $s$.  For $s \geq 2^{i+1}$ this holds since by equation~\eqref{e:initialYalpha1} we have that
\[
Y_{T+2^{i}}(s) \geq \tilde{y}(s) \geq s^{1/2}\log_2 s \geq y_\alpha(s).
\]
For $0 \leq s \leq 2^{i}$, on the event $\sD$,
\[
Y_{T+2^i}(s) \geq \lfloor s 2^{-2i/3} \rfloor 2^{2i/3} \frac1{40} 2^{-i/2} \geq \max\{0,s-\alpha^{-3/2}\} \frac1{40} 2^{-i/2} \geq y_{\alpha}(s),
\]
and for $2^i \leq s \leq 2^{i+1}$
\[
Y_{T+2^i}(s) \geq Y_{T+2^i}(2^i) \geq 2^{i} \cdot \frac1{40} 2^{-i/2} \geq  y_{\alpha}(2^{i+1}).
\]
Thus for all $s\geq 0$, $Y_{T+2^{i}}(s) \geq y_{\alpha}(s)$ which completes the proof.
\end{proof}
\begin{lemma}\label{l:maintainYalpha}
For all $K>1$, there exists $\Delta(K)$ and $\chi(K)>0$ such that if $0\leq \alpha \leq \Delta$ and $\inf_s Y_T(s) - y_\alpha(s) = 0$ then
\[
\P\Big[\sR\big(T,\alpha^{-4/3},  \frac{ \alpha(K + 1)}{2}\big)^c\mid Y_T\Big] \leq \exp\Big(-\chi(K) \alpha^{-1/3}\Big).
\]
\end{lemma}
\begin{proof}
With $\alpha_\star(\epsilon)$ defined as in Lemma~\ref{l:RWYalpha} set $\Delta(K)=\alpha_\star(\frac{K-1}{3K})$.  Then for $0\leq s \leq \alpha^{-4/3}$
\begin{align*}
S(T+t) &= \frac{K}{2}\P[\max_{0\leq s \leq t} W_s - Y_{T+t}(s) \leq 0\mid Y_{T+t}]\\
&\geq \frac{K}{2}\P[\max_{s \geq 0} W_s - y_\alpha((s-\alpha^{4/3})\wedge 0) \leq 0]\\
&\geq \frac{K}{2} 2\big(1-\frac{K-1}{3K}\big)\alpha = \frac{ \alpha(2K + 1)}{3}
\end{align*}
where the first inequality follows from the fact that
\[
Y_{T+t}(s) \geq Y_T(s-\alpha^{4/3})\wedge 0) \geq y_\alpha(s-\alpha^{4/3})\wedge 0)
\]
and the second inequality follows from Lemma~\ref{l:RWYalpha}.  Now take $\rho = \frac{3K+3}{4K+2}<1$ and with $\psi$ defined in Lemma~\ref{l:minSpeed} set $\chi(K) = \psi(\rho)$.  Then since
\[
\inf_{0\leq t \leq \alpha^{-4/3}}  S(T+t)\geq \frac{ \alpha(2K + 1)}{3} =\rho \frac{ \alpha(K + 1)}{2}
\]
by Lemma~\ref{l:minSpeed} we have that
\[
\P\Big[\sR\big(T,\alpha^{-4/3},  \frac{ \alpha(K + 1)}{2}\big)^c \mid Y_T\Big] \leq \exp\Big(-\chi(K) \alpha^{-1/3}\Big).
\]
\end{proof}

This result is useful because of the following claim.
\begin{claim}\label{c:rplusYalpha}
For some $0 \leq \alpha \leq \frac12$ suppose that $\inf_s Y_T(s) - y_\alpha(s) = 0$.  Then for an $0 \leq t \leq \alpha^{-3/2}$ and $\gamma \geq 1$ on the event $\sR\big(T,t,  \alpha \gamma\big)$ we have that $\inf_s Y_{T+t}(s) - y_\alpha(s) = 0$.
\end{claim}
\begin{proof}
Since $y_\alpha(s)=$ for $0\leq s \leq \alpha^{-3/2}$ it is sufficient to check $s \geq \alpha^{-3/2}$.  Then
\begin{align*}
Y_{T+t}(s) &= Y_T(s-t) + X_{T+t} - X_t\\
&\geq Y_T(s-t) + \alpha \gamma t\\
&\geq y_\alpha(s-t) + \alpha \gamma t \\
&\geq y_\alpha(s) - \alpha t + \alpha \gamma t \geq y_\alpha (t),
\end{align*}
where the first inequality is by the event $\sR\big(T,t,  \alpha \gamma\big)$, the second is by assumption and the third is since $\frac{d}{ds} y_\alpha(s)$ is uniformly bounded above by $\alpha$.
\end{proof}

\begin{lemma}\label{l:progressiveAlpha}
For all $K>1$, there exists $\Delta(K)$ and $\chi(K)>0$ such that if $0\leq \alpha \leq \Delta$ and $\inf_s Y_T(s) - y_\alpha(s) = 0$ then
\[
\P\Big[\inf_s Y_{T+\alpha^{-3}}(s) - y_{\frac{\alpha(K + 1)}{2} }(s) =0 \mid Y_T\Big] \leq \alpha^{-5/3}\exp\Big(-\chi(K) \alpha^{-1/3}\Big).
\]
\end{lemma}

\begin{proof}
Let $\sD_\ell$ denote the event,
\[
\sD_\ell=\sR(T+\ell \alpha^{-4/3},\alpha^{-4/3},\frac{\alpha(K + 1)}{2}).
\]
By Claim~\ref{c:rplusYalpha} and induction if $\bigcap_{\ell'=0}^{\ell-1} \sD_{\ell'}$ holds then $\inf_s Y_{T+\ell \alpha^{-4/3}}(s) - y_\alpha(s) = 0$.  Thus by Lemma~\ref{l:maintainYalpha} we have that
\[
\P\bigg[\sD_\ell \mid \bigcap_{\ell'=0}^{\ell-1} \sD_{\ell'}, Y_T\bigg] \geq 1 - \exp\Big(-\chi(K) \alpha^{-1/3}\Big)
\]
and so with $\sD^*=\bigcap_{\ell=0}^{\alpha^{-5/3}-1} \sD_{\ell}$,
\[
\P\big[\sD^* \mid  Y_T\big] \geq 1 - \alpha^{-5/3}\exp\Big(-\chi(K) \alpha^{-1/3}\Big).
\]
Now suppose that the event $\sD^*$ holds and assume that $\Delta(K)$ is small enough so that for all $0 \leq \alpha \leq \Delta(K)$ the following hold:
\begin{itemize}
\item $\alpha^{-4/3} \leq (\frac{\alpha(K + 1)}{2})^{-3/2}$,
\item $\frac{K + 1}{2} \alpha^{-2} \geq (2\alpha^{-3})^{1/2} \log_2 (2\alpha^{-3})$,
\item $\forall s \geq \alpha^{-3}, \ \min\{\alpha (s-\alpha^{-3/2}), s^{1/2}\log_2 s\} = s^{1/2}\log_2 s$,
\item $\forall s \geq \alpha^{-3}, \ \min\{\frac{\alpha(K + 1)}{2} (s-(\frac{\alpha(K + 1)}{2})^{-3/2}), s^{1/2}\log_2 s\} = s^{1/2}\log_2 s$,
\item $\inf_{s \geq 2\alpha^{-3}} - s^{1/2}\log_2 s + (s- \alpha^{-3})^{1/2}\log_2 (s- \alpha^{-3})  +  \frac{K + 1}{2} \alpha^{-2} \geq 0$.
\end{itemize}
It is straightforward to check that all of these hold for sufficiently small $\alpha$.
For all $(\frac{\alpha(K + 1)}{2})^{-3/2} \leq s \leq \alpha^{-3}$ that
\[
Y_{T+\alpha^{-3}}(s) \geq  \lfloor s \alpha^{4/3} \rfloor \alpha^{-4/3} \cdot \frac{\alpha(K + 1)}{2} \geq \left(s - \left(\frac{\alpha(K + 1)}{2}\right)^{3/2} \right)\frac{\alpha(K + 1)}{2} \geq y_{\frac{\alpha(K + 1)}{2}}(s).
\]
For $\alpha^{-3} \leq s \leq 2 \alpha^{-3}$,
\begin{align*}
Y_{T+\alpha^{-3}}(s) &\geq Y_{T}(s- \alpha^{-3}) + \frac{K + 1}{2} \alpha^{-2} \\
&\geq y_{\alpha}(s- \alpha^{-3}) + \frac{K + 1}{2} \alpha^{-2} \\
&\geq (2\alpha^{-3})^{1/2} \log_2 (2\alpha^{-3}) = y_{\frac{\alpha(K + 1)}{2}}(2 \alpha^{-3}).
\end{align*}
Finally, for $s\geq 2 \alpha^{-3}$,
\begin{align*}
Y_{T+\alpha^{-3}}(s) &\geq  y_{\alpha}(s- \alpha^{-3}) + \frac{K + 1}{2} \alpha^{-2}\\
&= y_{\frac{\alpha(K + 1)}{2}}(s) - s^{1/2}\log_2 s + (s- \alpha^{-3})^{1/2}\log_2 (s- \alpha^{-3})  +  \frac{K + 1}{2} \alpha^{-2}\\
&\geq y_{\frac{\alpha(K + 1)}{2}}(s).
\end{align*}
Combining the previous 3 equations implies that $Y_{T+\alpha^{-3}}(s) \geq y_{\frac{\alpha(K + 1)}{2}}(s)$ for all $s$ and hence
\[
\P\Big[\inf_s Y_{T+\alpha^{-3}}(s) - y_{\frac{\alpha(K + 1)}{2} }(s) =0 \mid Y_T\Big]\leq \P[\sD^*] \leq \alpha^{-5/3}\exp\Big(-\chi(K) \alpha^{-1/3}\Big).
\]
\end{proof}

\begin{lemma}\label{l:permissive2}
For all $K>1$, there exists $i^*(K)$ such that the following holds.  If $i\geq i^*$ and $Y_T$ is permissive for all $i'>i$ then
\[
\P\bigg[\min_{s\in[4^{i},2e^{2^{i/10}}]}Y_{T+s}(2^{i}) \leq 10 i 2^{i/2} \mid \cF_T\bigg] \leq 3e^{-2^{i/10}},
\]
that is $Y_{T+s}$ is permissive at scale $i$ for all $s\in [2^{i},2e^{2^{i/10}}]$.
\end{lemma}
\begin{proof}
We choose $i^*(K)$ large enough so that,
\[
20 i^* K 2^{-i^*/2} \leq \Delta(K)
\]
where $\Delta(K)$ was defined in \ref{l:progressiveAlpha}.
Set $t_0=2^i$ and $\alpha_0=\frac1{80} 2^{-i/2}$.  We define $\alpha_\ell=\left(\frac{K+1}{2}\right)^\ell \alpha_0$ and $t_\ell = t_{\ell-1} + \alpha_{\ell-1}^{-3}$.  Define the event $\sW_\ell$ as
\[
\sW_\ell = \Big\{ \inf_s Y_{T+t_\ell}(s) - y_{\alpha_\ell }(s) =0 \Big\}.
\]
By Lemma~\ref{l:initialYalpha} we have that
\[
\P[\sW_0\mid \cF_T] \geq 1-\exp(-2^{i/10}),
\]
and by Lemma~\ref{l:progressiveAlpha} we have that
\[
\P\bigg[\sW_\ell \mid \bigcap_{\ell'=0}^{\ell-1} \sW_{\ell'}\mid \cF_T\bigg] \geq 1- \alpha_{\ell-1}^{-5/3}\exp\Big(-\chi(K) \alpha_{\ell-1}^{-1/3}\Big).
\]
Now choose $L$ to be the smallest integer such that $\alpha_L \geq 20 i 2^{-i/2}$. So $L = \lceil \frac{\log(1600i)}{\log((K+1)/2)} \rceil$ which is bounded above by $i$ provided that $i^*(K)$ is sufficiently large.  Thus
\begin{align*}
\P[\sW_L\mid \cF_T] &\geq 1 - \exp(-2^{i/10}) - \sum_{\ell=0}^{L-1} \alpha_{\ell-1}^{-5/3}\exp\Big(-\chi(K) \alpha_{\ell-1}^{-1/3}\Big)\\
&\geq 1 - \exp(-2^{i/10}) - i(20 i 2^{-i/2})^{-5/3}\exp\Big(-\chi(K) (20 i 2^{-i/2})^{-1/3}\Big)\\
& \geq 1 - 2\exp(-2^{i/10})
\end{align*}
where the final inequality holds for $i$ is sufficiently large.  Now let $\sD_k$ denote the event,
\[
\sD_k=\sR(T+t_L+k \alpha_L^{-4/3},\alpha_L^{-4/3},\alpha_L).
\]
By Claim~\ref{c:rplusYalpha} on the event $\sW_L$ and $\bigcap_{k'=0}^{k-1} \sD_{k'}$ we have
\[
\inf_s Y_{T+t_L +k \alpha_L^{-4/3}}(s) - y_{\alpha_L }(s) =0.
\]
Thus by Lemma~\ref{l:maintainYalpha} we have that
\[
\P[\sD_k\mid \cF_T, \sW_L,  \bigcap_{k'=0}^{k-1} \sD_{k'}] \geq 1- \exp\Big(-\chi(K) \alpha_L^{-1/3}\Big).
\]
Let $\sD^*$ be the event
\[
\sD^* = \left\{ \sW_L,  \bigcap_{k'=0}^{e^{2^{i/10}}-1} \sD_{k'} \right\}.
\]
Then for $i$ sufficiently large since $\alpha_L \leq 20 i K 2^{-i/2}$,
\[
\P[ \sD^* \mid \cF_T] \geq 1- 2\exp(-2^{i/10}) - \exp\Big(2^{i/10}-\chi(K) \alpha_L^{-1/3}\Big) \geq 1 -  3\exp(-2^{i/10}).
\]
One the event $\sD^*$ we have that for all $t_L+2^i \leq s \leq \alpha_L^{-4/3} e^{2^{i/10}}$ that
\[
Y_{T+s} \geq  \alpha_L (2^i - \alpha_L^{-4/3}) \geq 10 i 2^{i/2}.
\]
By construction $t_L = 2^i +\sum_{\ell=0}^{L-1} \alpha_\ell^{-3} \leq 4^i$
and hence
\[
\P\bigg[\min_{s\in[4^{i},2e^{2^{i/10}}]}Y_{T+s}(2^{i}) \leq 10 i 2^{i/2} \mid \cF_T\bigg]\leq \P[ (\sD^*)^c \mid \cF_T] \leq 3e^{-2^{i/10}} \, .
\]
\end{proof}

\begin{corollary}\label{c:initialPermissive}
For all $K>1$, there exists $i^*(K)$ such if $i\geq i^*$ then
\[
\P\bigg[\min_{s\in[0,e^{2^{i/10}}]}Y_{s}(2^{i}) \leq 10 i 2^{i/2} \bigg] \leq 3e^{-2^{i/10}},
\]
\end{corollary}
\begin{proof}
We can apply Lemma~\ref{l:permissive2} to time $T=0$ since it is permissive at all levels and hence have that
\[
\P\bigg[\min_{s\in[4^{i},2e^{2^{i/10}}]}Y_{s}(2^{i}) \leq 10 i 2^{i/2} \bigg] \leq 3e^{-2^{i/10}},
\]
Since $Y_t$ is stochastically decreasing in $t$ we have that
\[
\P\left[\min_{0\leq t \leq e^{2^{i/10}}} Y_t(2^i) \leq  10 i 2^{i/2} - i 2^{i}\right] \leq  \P\bigg[\min_{s\in[4^{i},4^i + e^{2^{i/10}}]}Y_{s}(2^{i}) \leq 10 i 2^{i/2} \bigg] \leq 3e^{-2^{i/10}},
\]
which completes the corollary.
\end{proof}

\begin{lemma}\label{l:permissive3}
For all $K>1$, there exists $i^*(K)$ such that
\[
\inf_t \P\bigg[\forall i \geq i^*, Y_{t}(2^{i}) \geq 10 i 2^{i/2}\bigg] \geq \frac12.
\]
\end{lemma}
\begin{proof}
Take $i^*(K)$ as in Lemma~\ref{l:permissive2} and suppose that $I \geq i^*$.  Let $\sD_{I}$ denote the event that  $Y_t$ is permissive for all levels $i\geq I$ and all $t\in [0,e^{2^{I/10}}]$.  By Corollary~\ref{c:initialPermissive} we have that
\[
\P[\sD_{I}^c] \leq \sum_{i\geq I} 3 e^{-2^{i/10}} \leq 4 e^{-2^{I/10}}.
\]
Next set $t_0=\frac12 e^{2^{I/10}}$   and let $t_k=t_{k-1}+4^{I-k}$.  Let $\sH_k$ denote the event that $Y_t$ is permissive at level $I-k$ for all $t\in[t_k,t_k+e^{2^{(I-k)/10}}]$.  By Lemma~\ref{l:permissive2} then for $0\leq k \leq I-i^*$,
\[
\P[\sH_k^c,\cap_{k'=1}^{k-1} \sH_{k'}, \sD_{i^*}] \leq 3 e^{-2^{(I-k)/10}}.
\]
Thus, provided $i^*$ is large enough,
\[
\P[\cap_{k'=1}^{I-i^*} \sH_{k'}, \sD_{i^*}] \geq 1 - 4 e^{-2^{I/10}} - \sum_{k'=1}^{I-i^*} 3e^{-2^{(I-k)/10}} \geq \frac12.
\]
Let $\tau=\tau_I=t_{I-i^*}$.  Then for all $I \geq i^*$,
\[
\P\bigg[\forall i \geq i^*, Y_{\tau_I}(2^{i}) \geq 10 i 2^{i/2}\bigg] \geq \delta.
\]
since $Y_t$ is stochastically decreasing in $t$ and $\tau_I\to \infty$ and $I\to \infty$,
\[
\inf_t \P\bigg[\forall i \geq i^*, Y_{t}(2^{i}) \geq 10 i 2^{i/2}\bigg] \geq \frac12.
\]
\end{proof}

\begin{theorem}\label{t:convProb}
For $K >1 $ there exists a random function $Y^*(s)$ such that $Y_t$ converges weakly to $Y^*$ in finite dimensional distributions.  Furthermore,
with
\[
\alpha^*= \frac{K}2\E\left[ \P[\max_{0\leq s \leq t} W_s - Y^*(s) \leq 0\mid Y^*]\right],
\]
we have that $\frac1{t} X_t$ converges in probability to $\alpha^* >0$.
\end{theorem}
\begin{proof}
Since $Y_t$ is stochastically decreasing it must converge in distribution to some limit $Y^*$.
By Claim~\ref{c:SLowBound}
\[
\P\left[\frac{K}{2} \P[\max_{0\leq s \leq t} W_s - Y^*(s) \leq 0\mid Y^*] \geq \frac{K}{10}2^{-i/2}\prod_{i=i^*}^\infty (1-e^{1-\max\{1,10i/\sqrt{2}\}})\right] \geq \frac12,
\]
and so $\alpha^* =\lim_t \E S(t) >0$.
To show convergence in probability fix $\epsilon >0$.  For some large enough~$L$,
\[
\E[\frac1{L} X_L] = \frac1{L} \int_{0}^{L}  \E S(t) dt \leq \alpha^* + \epsilon/2.
\]
Let $N_k= \E[X_{kL}-X_{(k-1)L}\mid \cF_{(k-1)L}]$ and $R_k = X_{kL}-X_{(k-1)L} - N_k$.  By monotonicity
\[
N_k \leq \E[\frac1{L} X_L] \leq \alpha^* + \epsilon/2.
\]
The sequence $R_k$ are martingale differences with uniformly bounded exponential moments (since it is bounded from below by $-(\alpha^* + \epsilon/2)$ and stochastically dominated by a Poisson with mean $LK$).  Thus
\[
\lim_n \frac1{n}\sum_{k=1}^n R_k  = 0 \hbox{ a.s.} \ .
\]
It follows that almost surely $\limsup_t \frac1{t} X_t \leq \alpha^*$.  Since $X_t$ is stochastically dominated by $\hbox{Poisson}(Kt)$ we have that $\E[(\frac1{t} X_t)^2] \leq K^2 + K/t$ and so is uniformly bounded.  Hence since $\lim \E \frac1{t} X_t \to \alpha^*$ it follows that we must have that $\frac1{t} X_t$ converges in distribution to $\alpha^*$.
\end{proof}

\section{Regeneration Times}
In order to establish almost sure convergence to the limit we define a series of regeneration times.  We select some small $\alpha(K)>0$, and say an integer time $t$ is a regeneration time if
\begin{enumerate}
\item The function $Y_t$ satisfies $\inf_s Y_t(s) - y_{\alpha}(s) = 0$.
\item For $J_t$ the set of particles to the right of the aggregate at time $t$, their trajectories~$\{\zeta_j(s)\}_{j\in J_t}$ on $(-\infty,t]$ satisfy
\[
\inf_s \zeta_j(t-s) - (X_t - y_{\alpha}(s))> 0.
\]
\end{enumerate}
Let $0\leq T_1<  T_2< \ldots$ denote the regeneration times and let $\fR$ denote the set of regeneration times.
\begin{lemma}\label{l:regen}
For all $K>1$, there exists $\delta(K)>0$ such that,
\[
\inf_{t\in \mathbb{N}} \P[t\in \fR] \geq \delta.
\]
\end{lemma}
\begin{proof}
Let $\sD_t$ be the event that $\inf_s Y_t(s) - y_{\alpha}(s) = 0$.
Provided that $\alpha(K)$ is small enough by Lemmas~\ref{l:initialYalpha} and~\ref{l:permissive3} we have that
\[
\P[\sD_t] \geq \frac13.
\]
As the density of particles to the right of $X_t$ is increasing in $Y_t$ it is, therefore greatest when $t=0$ and so $\P[t\in \fR \mid \sD_t]$ is minimized at $t=0$.  Let $w_\ell$ be defined as the probability
\[
w_\ell= \P[\max_{0\leq s \leq t} W_s - y_\alpha(s) > \ell]
\]
For $0\leq \ell < \alpha^{-4}$ we simply bound $w_\ell \leq 1$ so let us consider $\ell \geq \alpha^{-4}$.  Then
\begin{align*}
w_\ell &\leq  1 - \P[M_{\ell}\leq \ell, \forall i \geq \lfloor \log_2 (\ell) \rfloor : \ M_{2^{i+1}} \leq \ell + i 2^{i+1}]\\
&\leq 1 - \P[M_{\ell}\leq \ell]\prod_{i \geq \lfloor \log_2 (\ell) \rfloor} \P[M_{2^{i+1}} \leq \ell + i 2^{i+1}]\\
&\leq 1 - (1 - e^{1-\ell^{1/2}})\prod_{i \geq \lfloor \log_2 (\ell) \rfloor} (1-e^{1-i/\sqrt{2}})\\
&\leq e^{1-\ell^{1/2}} + \sum_{i \geq \lfloor \log_2 (\ell) \rfloor} e^{1-i/\sqrt{2}}
\end{align*}
where the third inequality is be the FKG inequality and the final inequality is by equation~\eqref{e:reflectionBound}.
Then we have that
\begin{align*}
\sum_{\ell \geq \alpha^{-4}} w_\ell &\leq \sum_{\ell \geq \alpha^{-4}} e^{1-\ell^{1/2}} + \sum_{\ell \geq \alpha^{-4}} \sum_{i \geq \lfloor \log_2 (\ell) \rfloor} e^{1-i/\sqrt{2}}\\
& \leq  \sum_{\ell \geq \alpha^{-4}} e^{1-\ell^{1/2}} + \sum_{i} 2^{i+1} e^{1-i/\sqrt{2}} < \infty,
\end{align*}
since $2 e^{-1/\sqrt{2}}<1$.  Hence $\sum_{\ell =0}^\infty w_\ell < \infty$ and so
\[
\P[0\in \fR \mid \sD_0] = \P[\hbox{Poisson}(K\sum_{\ell =0}^\infty w_\ell)=0]>0.
\]
Thus there exists $\delta>0$ such that $\inf_{t\in \mathbb{N}} \P[t\in \fR] \geq \delta$.
\end{proof}

We can now establish our main result.

\begin{proof}[Theorem~\ref{t:mainThm}]
By Lemma~\ref{l:regen} there is a constant density of regeneration times so the expected inter-arrival time is finite.  By Theorem~\ref{t:convProb} the process $X_t$ travels at speed $\alpha^*$, at least in probability.  By the Strong Law of Large Numbers for renewal-reward processes this convergence must also be almost sure.
\end{proof}

\section{Higher dimensions}\label{s:Higher}

Our approach gives a simple way of proving positive speed in higher dimensions as well although not down to the critical threshold.  Simulations for small $K$ in two dimensions produce pictures which look very similar to the classical DLA model.  Surprisingly, however, Eldan~\cite{Eldan:16} conjectured that the critical value for $d\geq 2$ is 0!  That is to say that despite the simulations there is linear growth in of the aggregate for all densities of particles and that these simulations are just a transitory effect reflecting that we are not looking at large enough times.  We are inclined to agree but our techniques will only apply for larger values of $K$.  A better understanding of the notoriously difficult classical DLA model may be necessary, for instance that the aggregate has dimension smaller than 2.

Let us now assume that $K > 1$.  In the setting of $\Z^d$ it will be convenient for the sake of notation to assume that the particles perform simple random walks with rate $d$ which simply speeds the process be a factor of $d$.  The projection of the particles in each co-ordinate is then a rate 1 walk.  We let $U_t$ be the location of the rightmost particle in the aggregate (if there are multiple rightmost particles take the first one) at time $t$ and let $X_t$ denote its first coordinate.  We then define $Y_t(s)$ according to~\eqref{e:yDefn} as before.  We call a particle with path $\big(Z_1(t),\ldots,Z_d(t)\big)$  \emph{conforming} at time $t$ if $Z_1(s) > X_s$ for all $s\leq t$. By construction conforming particles cannot be part of the aggregate and conditional on $X_t$ form a Poisson process with intensity depending only on the first coordinate.

Let $e_i$ denote the unit vector in coordinate $i$.  The intensity of conforming particles at time $t$ at $U_t + e_1$ is then simply
\[
K \P[\max_{0\leq s \leq t} W_s - Y_t(s) \leq 0\mid Y_t].
\]
where $W_s$ is an independent simple random walk.  Similarly the rate at which conforming particles move from $U_t + e_1$ to $Y_t$  thus forming a new rightmost particle is
\[
S(t) = \frac12 K \P[\max_{0\leq s \leq t} W_s - Y_t(s) \leq 0\mid Y_t],
\]
the same as the formula we found in the one dimensional case.  Of course by restricting to conforming particles we are restricting ourselves and so the rate at which $X_t$ increments is strictly larger than $S(t)$.  Since $S(t)$ is increasing as a function of $X_t$ (through $Y_t$) we can stochastically dominate the one dimensional  case by the higher dimensional process which establishes Corollary~\ref{c:Higher}.

Let us now briefly describe how to improve upon $K=1$.  In the argument above we are being wasteful in two regards, first by only considering conforming particles and secondly by considering only a single rightmost particle.  If there are two rightmost particles then the rate at which $X_t$ increases doubles.  The simplest way to get such a new particle is for a conforming particle at $U_t +e_1 \pm e_i$ to jump first to $U_t  \pm e_i$ and then to $U_t$. There are $(2d-2)$ such location and the first move occurs at rate $S(t)$ and the second at has probability $1/(2d)$ to move in the correct direction and takes time exponential with rate $d$.  After this sequence of events the rate at which $X_t$ increments becomes $2S(t)$.

In Lemma~\ref{l:minSpeed}, on which the whole proof effectively rests, we show that for $\rho<1$ if $S(s) \geq \gamma$ for $s\in [t,t+\Delta]$ then with exponentially high probability $X_{t+\Delta} - X_t \geq \rho \Delta \gamma$ for any $\rho < 1$  which is intuitively obvious since $X_t$ grows at rate $S(s) \geq \gamma$.  We can improve our lower bound on $K$ by increasing the range of $\rho$ for which this holds for small values of $\gamma$.

Define the following  independent random variables
\[
V_1\sim \hbox{Exp}(\gamma),V_2\sim \hbox{Exp}(\frac{2d-2}{2d}\gamma), V_3\sim \hbox{Exp}(d), V_4\sim \hbox{Exp}(\gamma)
\]
where we interpret $V_1$ as the time until the first conforming particle hits $U_t$.  We will view $V_2$ as the waiting time for a conforming particle to move from $U_t + \pm e_i + e_1$ to  $U_t \pm e_i$ for some $2\leq i \leq d$ and we further specify that their next step will move directly to $U_t$ which thins the process by a factor $\frac1{2d}$.
Let $V_3$ be the time until its next move. On the event $V_2+V_3 < V_1$ there is an additional rightmost particle before one has been added to the right of $U_t$.  Now let $V_4$ be the first time a conforming particle reaches this new rightmost site.  So the time for $X_t$ to increase is stochastically dominated by
\[
T= \min\{V_1 + V_2+V_3 + V_4 \}.
\]
Now using the memoryless property of exponential random variables,
\begin{align*}
\E T  &= \E V_1 - \E\left[(V_1 - (V_2+V_3 + V_4))I(V_1 \geq V_2+V_3 + V_4)\right]  = \frac1{\gamma}(1-\P[V_1 \geq V_2+V_3 + V_4])
\end{align*}
and
\begin{align*}
\P[V_1 \geq V_2+V_3 + V_4] &= \P[V_1 \geq V_2]\P[V_1 \geq V_2+V_3 \mid V_1 \geq V_2] \P[V_1 \geq V_2+V_3 + V_3 \mid V_1 \geq V_2+V_3]\\
&= \frac{\frac{2d-2}{2d}\gamma}{\gamma + \frac{2d-2}{2d}\gamma} \frac{d}{\gamma+d}\frac{\gamma}{2\gamma}\\
&= \frac{d-1}{2(2d-1)} \frac{d}{\gamma+d}\\
\end{align*}
In the proof we need only to consider the case where $\gamma$ is close to 0 and
\[
\lim_{\gamma\to 0} \gamma \E T = \frac{3d-1}{4d-2}.
\]
Having $X_t$ growing at rate $\gamma\frac{4d-2}{3d-1}$ corresponds in the proof to linear growth provided that $K > \frac{3d-1}{4d-2}$.  In the case for $d=2$ this means $K > \frac{5}{6}$.  We are still being wasteful in several ways and  expect that a more careful analysis would yield better bounds that tend to 0 as $d\to \infty$.  However, we don't believe that this approach alone is sufficient to show that the critical value of $K$ is 0 when $d\geq 2$.  For that more insight into the local structure is likely needed along with connections to standard DLA.

\section{Open Problems}

In the one dimensional case the most natural open questions concern the behaviour of $X_t$ for densities close to 1.  Approaching $K=1$ from above one can ask what exponent does the speed of the process satisfy.  Perhaps of most interest is what is the exponent of growth for $X_t$ when $K=1$.  Heuristics suggest that it may grow as $t^{2/3}$.

In higher dimensions the main open problem is to establish Eldan's conjecture of linear growth for all $K$.  Another natural question is to prove a shape theorem for the aggregate.

{\bf Acknowledgements:}  The author would like to that Vladas Sidoravicius for useful discussions and for his comments on a draft of the paper and to NYU Shanghai for their hospitality.

\bibliographystyle{plain}
\bibliography{DlaBib}
\end{document}